\documentclass[11pt]{article}
\usepackage{graphicx} 
\usepackage{amsmath,amssymb,amsthm}
\usepackage{mathtools,url}
\usepackage{graphicx}
\usepackage{comment}
\usepackage[dvipsnames]{xcolor}

\usepackage{letltxmacro}
\makeatletter
\let\oldr@@t\r@@t
\def\r@@t#1#2{%
\setbox0=\hbox{$\oldr@@t#1{#2\,}$}\dimen0=\ht0
\advance\dimen0-0.2\ht0
\setbox2=\hbox{\vrule height\ht0 depth -\dimen0}%
{\box0\lower0.4pt\box2}}
\LetLtxMacro{\oldsqrt}{\sqrt}
\renewcommand*{\sqrt}[2][\ ]{\oldsqrt[#1]{#2}}
\makeatother

\makeatletter
\DeclarePairedDelimiterX{\pmodx}[1]{(}{)}{{\operator@font mod}\mkern6mu#1}
\renewcommand{\pmod}{%
  \allowbreak
  \if@display\mkern18mu\else\mkern8mu\fi
  \pmodx
}
\makeatother

\theoremstyle{plain}
\newtheorem{theorem}{Theorem}
\newtheorem*{theorem*}{Theorem}
\newtheorem{lemma}{Lemma}
\newtheorem{corollary}{Corollary}
\newtheorem*{corollary*}{Corollary}
\newtheorem{definition}{Definition}

\theoremstyle{definition}

\newtheorem*{remark*}{Remark}

\newcommand{\+}[1]{{#1}^\times}
  
\def\Aff{\operatorname{Aff}}
\DeclareMathOperator{\Ker}{Ker}
\DeclareMathOperator{\Ran}{Ran}

\DeclareMathOperator{\ord}{ord}
\DeclareMathOperator{\lcm}{lcm}

\DeclareMathOperator{\rad}{rad}

\def\R{{\mathbb R}}

\def\Z{{\mathbb Z}}
\def\C{{\mathbb C}}

\def\H{\widetilde{H}}

\def\mD{\overline{D}}
\def\sD{D^{\,\#}}

\def\a{{\alpha}}
\def\b{{\beta}}
\def\s{{\sigma}}
\def\k{{\varkappa}}
\def\pfi{{\varphi}}
\def\e{{\varepsilon}}

\def\id{{\rm id\,}}

\title{$N$-ary groups of panmagic permutations\\ from the Post coset theorem}
\author{Sergiy Koshkin$^*$ and Jaeho Lee$^\dagger$\\
\\
$^*$Corresponding author\\
Department of Mathematics and Statistics\\
University of Houston-Downtown\\
1 Main Street\\
Houston, TX 77002\\
e-mail: \texttt{koshkins@uhd.edu}
\\
\\
$^\dagger$Spring Branch Academic Institute\\
14400 Fern Drive\\
Houston, TX 77079\\
e-mail: \texttt{leejaeho0802@gmail.com} 
}

\date{}

\begin{document}

\maketitle

\begin{abstract}
Panmagic permutations are permutations whose matrices are panmagic squares, better known as maximal configurations of non-attacking queens on a toroidal chessboard. Some of them, affine panmagic permutations, can be conveniently described by linear formulas of modular arithmetic, and we show that their sets are a generalization of groups with $N$-ary multiplication instead of binary one. With the help of the Post coset theorem, we identify panmagic $N$-ary groups as cosets of the dihedral subgroup and its extensions in the group of all affine permutations. We also investigate decomposition of panmagic permutations into disjoint cycles and find many connections with classical topics of number theory and combinatorics: square-free numbers, $4k+1$ primes, quadratic residues, cycle indices from Polya counting, and linear congruential generators.
\medskip

\medskip

\end{abstract}

\section{Introduction}

Magic squares are square matrices with the same sum, called the {\it magic sum}, in each row, each column, the main diagonal and the main anti-diagonal. They attracted attention of mathematicians since Euler's 1776 article linked them to Latin squares \cite{BS2}. Magic squares that additionally have the same magic sum over all `broken' diagonals and anti-diagonals were once called ``diabolic" \cite{Lehm}, but now are commonly known as pandiagonally magic or {\it panmagic} for short \cite{AlKi,BS1}. 

Back in 1950, Lawden claimed that in even dimensions multiplying any three panmagic squares produces another panmagic square \cite{Lawd}. In fact, Lawden's result only holds for panmagic squares with additional symmetry, but it was the first result (that we know of) on closure under ternary multiplication for a class of magic squares. Other classes of magic squares with the same property were discovered in 1990s. Thompson might have been the first to prove it in print for $3\times3$ magic and $5\times5$ panmagic squares in 1994 \cite{Thom}. In 2012, known classes of magic and panmagic squares closed under ternary multiplication were codified and generalized to higher dimensions by Nordgren \cite{Nord}. 

Thompson's $5\times5$ example stood out from the rest. He found a spanning set of $5\times5$ panmagic permutation matrices that is itself closed under ternary multiplication, while no such spanning set exists for $3\times3$ magic or Lawden's panmagic squares. Moreover, this spanning set is closely associated with the dihedral group $D_5$, the group of symmetries of the regular pentagon. The nature of this association remained obscure, but it suggested to us looking into permutations with panmagic matrices, which are simpler objects than the matrices themselves.

Although quite natural, the term ``panmagic permutations" is rarely used (it is used in \cite{AlKi,KoLeeP}), but they are classically known under a different name. A problem from the mid-19th century, often misattributed to Gauss \cite{BS2}, asked to place $8$ non-attacking queens on a chessboard, and it was generalized to $n\times n$ chessboards by Lionnet in 1869. In 1900, Carpenter proposed to fold the board into a cylinder by identifying its two opposite sides \cite{Carp}, and Polya further folded it into a torus in 1918. Both modifications have the same effect of letting the queens attack along the entire broken diagonals, and gave rise to the {\it modular $n$-queens problem} \cite{BS2}. Given its solution, replace the board with a matrix and place $1$-s into the positions of the queens and $0$-s elsewhere as on Figure\,\ref{ChessPerm}. The result is a panmagic permutation matrix, and conversely, any panmagic permutation produces a modular $n$-queens solution. 
\begin{figure}[!ht]
\vspace{-0.3em}
\begin{centering}
\includegraphics[scale=0.18]{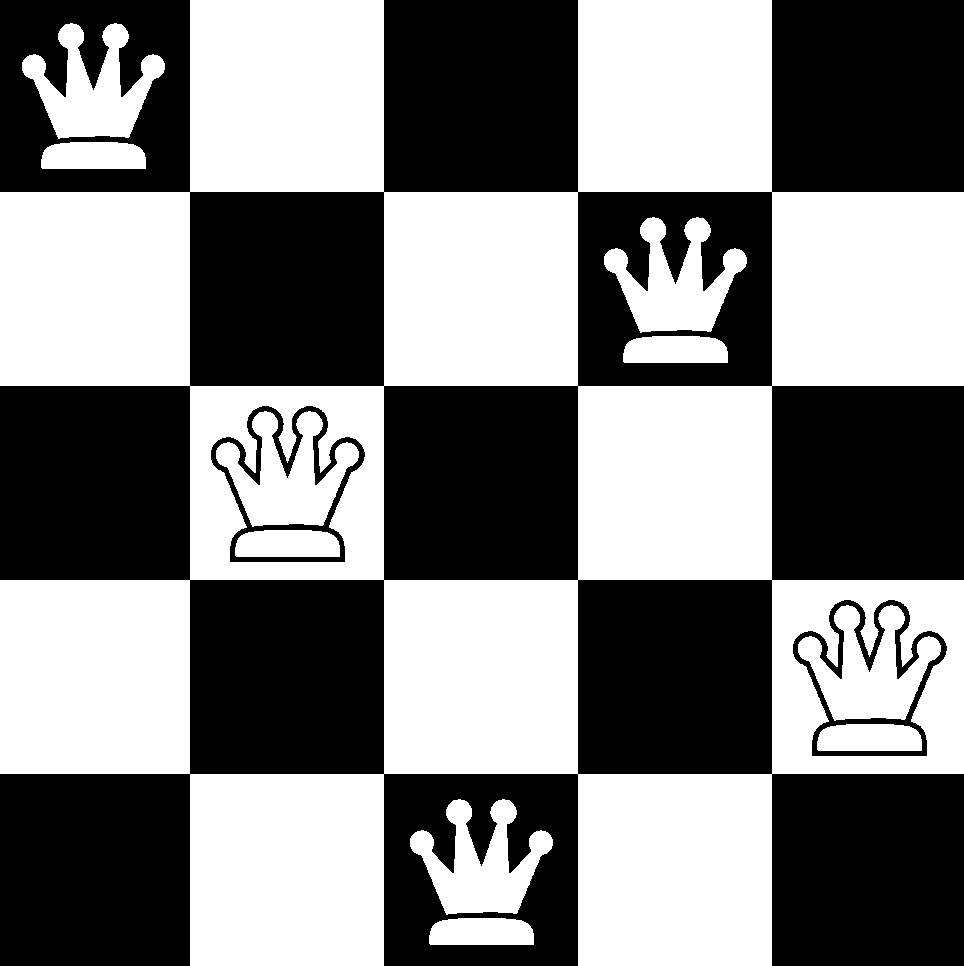}\hspace{3em} \includegraphics[scale=0.9]{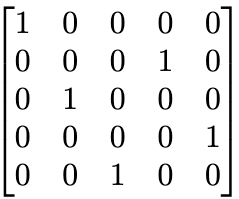}
\par\end{centering}
\caption{\label{ChessPerm} Modular $5$-queens solution and its permutation matrix.}
\end{figure}
\vspace{-0.5em}
In this guise, panmagic permutations have been extensively studied \cite{BS2}, but not, it seems, for their group-like properties. Even the connection between modular $n$-queens and panmagic squares is little known \cite{BS1} despite being used as early as in 1900 by Planck \cite{BS2}.

For that matter, ternary and more general $N$-ary groups are also little known, but the topic is, again, classical. It was originally developed by D\"ornte in 1920-s, at the urging of Emmy Noether \cite{GBV,Shah}. Post, better known for his work in mathematical logic, proved a key structural result about them in 1940, the Post coset theorem \cite{Post}. It says, essentially, that $N$-ary groups are cosets of normal subgroups when the quotient group is cyclic of order $N-1$, and explains the appearance of $D_5$ in Thompson's example. The group that contains the coset is called its {\it Post cover}, and the normal subgroup it is a coset of is called its {\it Post associate}. 

For $N$-ary groups of panmagic permutations, the Post covers are subgroups of the group of {\it affine permutations} $\Aff(\Z_n)$, those that can be expressed by linear formulas of modular arithmetic in $\Z_n$. And the Post associates are extensions of the dihedral group $D_n$ isomorphic to the group of symmetries of the regular $n$-gon. As we will prove, those turn out to be the groups of permutations that preserve the panmagic property when multiplied by panmagic permutations.

Using the Post coset theorem, we will show that affine panmagic $N$-ary subgroups of $S_n$ exist in all dimensions where panmagic permutations exist. Moreover, we will compute their minimal arities, characterize $n$ for which subgroups of a given arity exist, describe cycle types of permutations in them and subgroups where all permutations have the same cycle type. Even for the ternary case, our descriptions are much more explicit than Nordgren's in terms of linear algebra \cite{Nord}. Our results also relate affine panmagic permutations to some classical topics of number theory and combinatorics: square-free numbers, $4k+1$ primes, quadratic residues, cycle indices from Polya counting, and linear congruential generators \cite{BK-V,IrRos,Mars,Nath,Rosen}. Affine panmagic permutations and their $N$-ary groups for prime dimensions were studied in \cite{KoLeeP}, to which this paper is a follow-up.

The paper is organized as follows. In Section \ref{prelim} we introduce some definitions and notation, and in Section \ref{N-ary} basics of $N$-ary group theory that we need to describe panmagic $N$-ary subgroups. Section \ref{Affine} introduces the group of affine permutations, and  Section \ref{PMagPres} panmagic-preserving subgroups that will serve as their Post associates. Our application of the Post coset theorem is based on using suitable homomorphisms with cyclic ranges that are constructed in Section \ref{SlopeSq}. Our main results on $N$-ary subgroups of affine panmagic permutations are proved in Sections \ref{SqCyc},   \ref{StrictNary}, and on their cycle types in Section \ref{CycAffPan}. In the final Section \ref{Conc}, we draw our conclusions and outline some directions of future work.

\section{Preliminaries}\label{prelim}

This section introduces some notation and terminology used in the paper, mostly following \cite{Gallian,Nord,Rosen}. In particular, we will use standard notation and terminology of set theory and group theory \cite{Gallian}. 

The number of elements in a set $S$ will be denoted by $|S|$, products of sets in groups by elements and each other by $aS:=\{as\mid s\in S\}$ and $ST=\{st\mid s\in S\,,t\in T\}$. The subgroup of a group generated by elements of its subset $S$ will be denoted $\langle S\rangle$, and simply $\langle a_1,\dots,a_k\rangle$ when $S=\{a_1,\dots,a_k\}$.

The symmetric group of all permutations of $1,2,\dots,n$ is denoted  $S_n$, and its identity element is denoted $\id\!.$ We will identify $S_n$ with the group of permutations of $\Z_n=\{0,1,\dots n-1\}$, the set of residues (congruence classes) modulo $n$. To recover a standard permutation of $1,\dots,n$, one simply replaces residue $0$ by $n$. 

Two special permutations will play a central role in our study: the flip permutation $\phi$ that swaps antipodal numbers, i.e. $\phi(i):=n+1-i$, and the cyclic shift permutation $\k$ that shifts numbers one unit up and $n$ to $1$, $\k(i)\equiv i+1\pmod{n}$. In the cyclic notation, $\phi=(1\,n)(2\,n-1)\dots$ and $\k=(1\,2\dots n)$. Both $\phi$ and $\k$ depend on $n$, but we omit $n$ from notation as it will be clear from the
context.

The subgroup $D_n:=\langle\phi,\k\rangle$ of $S_n$ will be called {\it dihedral group} and its elements {\it dihedral permutations} (note that some authors denote this group $D_{2n}$ by the number of elements in it). If we number the vertices of a regular $n$-gon from $1$ to $n$ then  $\phi$ is the permutation induced by the reflection about the middle perpendicular of one of its sides, and $\k$ is induced by the rotation by the angle $\frac{2\pi}{n}$ around its center. Recall that in $D_n$ the following relations hold: $\phi^2=\k^{n}=\id$ and $\phi\k=\k^{-1}\phi$. They imply that all dihedral permutations are of the form $\k^i$ or $\phi\k^i$ with $i=0,\dots,n-1$.

Congruence formulas for residues will often be written in a simplified notation that omits the modulus $n$ when it is understood, and replaces $\equiv$ by $=$\,. For example, we may write $n\equiv0\pmod{n}$ as just $n=0$ when no confusion results. In our simplified notation, $\phi(i)=1-i$ and $\k(i)=i+1$. The same notational convention applies to matrix indices when they go over $n$ or under $1$. For instance, in an $n\times n$ matrix $a_{n+1\,1}$ is $a_{11}$. The {\it broken diagonals} of an $n\times n$ matrix $A$ contain those entries $a_{ij}$ where $i-j$ is a fixed constant, and the {\it broken anti-diagonals} contain those where $i+j$ is a fixed constant. The {\it main diagonal/anti-diagonal} is the one with that constant equal to $0$.

A square matrix $A$ is called {\it semimagic} when all of its column sums and row sums are the same. Their common value is called the {\it magic sum}. We will be mainly interested in matrices $P_\s$ with $\s\in S_n$ that permute the standard basis vectors $e_i$ as in $P_\s e_i:=e_{\s(i)}$. Clearly, all permutation matrices are semimagic with the magic sum $1$. Moreover, $P_{\s\tau}=P_\s P_\tau$, where on the left we have the composition of permutations and on the right the product of matrices. 

A semimagic matrix (square) is called {\it magic} when its main diagonal and anti-diagonal sums are also equal to the magic sum. And it is called {\it panmagic} when its broken diagonal and broken anti-diagonal sums are all equal to the magic sum. A permutation will be called magic or panmagic when its matrix is such.

The effect of multiplying a matrix by $P_{\phi}$ is to swap $i$-th and $(n+1-i)$-th rows/columns. This also interchanges (broken) diagonals and anti-diagonals. In turn, multiplication by $P_{\k}^i=P_{\k^i}$ on the left/right cyclically shift its rows/columns by $i$ positions and moves the main diagonal to one of the broken diagonals. 

Note that the number of $1$-s on the diagonal of $P_\s$ is the number of {\it fixed points} of $\s$, those for which $\s(i)=i$. Similarly, since $\phi$ flips points about the middle  the number of $1$-s on the anti-diagonal of $P_\s$ is the number of {\it flipped points} of $\s$, those with $\s(i)=\phi(i)$. Thus, a permutation is magic when it has exactly one fixed and exactly one flipped point, and it is panmagic when all of its cyclic shifts $\k^i\s$ are also magic.

An elementary observation that goes back to Lionnet, in connection with the 19-th century $n$-queens problem, is that $\s\in S_n$ is panmagic if and only if $i\mapsto\sigma(i)\pm i$ are also both permutations \cite{AlKi,BS1}. And a classical theorem of Polya from 1918 states that $S_n$ contains panmagic permutations if and only if $2,3\nmid n$ \cite{BS1}.
A positive integer $n$ will be called {\it Polya dimension} when $n>1$ and $2,3\nmid n$.
The smallest Polya dimension is $n=5$ of Thompson's example, and we will assume that $n$ is Polya from now on, unless otherwise stated.

\section{$N$-ary subgroups}\label{N-ary}



$N$-ary groups can be defined abstractly, like ordinary groups, but we will only need $N$-ary {\it sub}groups of ordinary groups and those are easier to define. One can prove that any abstract $N$-ary group is isomorphic to an $N$-ary subgroup \cite{GBV,Shah}.  
\begin{definition}\label{NaryG}
A subset $C\subseteq G$ of a group $G$ is called its $N$-ary subgroup when for any $N$ elements $a_1,\dots,a_N\in C$, the product $a_1\cdots a_N\in C$; and for any $a\in C$, the element $a^{3-2N}\in C$. 
\end{definition} 
\noindent The first condition replaces closure under binary multiplication by its $N$-ary analog and the second replaces closure under taking inverses. When $N=2$ we recover the definition of ordinary (binary) subgroup. 

In the ternary case, $a^{3-2N}=a^{-3}$, but we still get closure under inversion because $a^{-1}=aaa^{-3}\in C$ as a triple product. In fact, for $N\geq3$ we can generally replace the condition $a^{3-2N}\in C$ with a simpler equivalent one $a^{2-N}\in C$. Indeed, $a^{2-N}=a^{N-1}a^{3-2N}$ and $a^{3-2N}=a^{N-3}\left(a^{2-N}\right)^3$, and both are $N$-ary products for $N\geq3$. The element $a^{2-N}$ is called {\it skew to $a$} and is typically used as the replacement for the inverse in the $N$-ary group theory \cite{GBV,Shah}. 

We did not say anything about the identity element $e$ and this is by design. In the binary case, $e$ will be in $C$ because we can multiply any element by its inverse. But also conversely, if $e\in C$ then $C$ is a binary subgroup because we can reproduce binary products by taking $N$-ary ones with $N-2$ copies of $e$ in them. Indeed, under Definition \ref{NaryG} any binary subgroup is also $N$-ary for any $N\geq3$. 

Nonetheless, there are non-binary $N$-ary subgroups for $N>2$. In fact, many of them are quite familiar, albeit rarely thought about this way. For example, odd integers form a ternary (additive) subgroup of $\Z$. Similarly, odd permutations form a non-binary ternary subgroup of $S_n$. Non-zero purely imaginary numbers $i\+{\R}$ form a (multiplicative) ternary subgroup of non-zero complex numbers $\+{\C}$. Indeed, $ia\cdot ib\cdot ic=i(-abc)$ and $(ia)^{-1}=i(-a)$. 

Thompson's ternary panmagic subgroup of $S_5$ \cite{Thom} can be described as 
\begin{equation}\label{psiD5}
\psi D_5=\{\psi,\psi\k,\dots,\psi\k^4,\psi\phi,\psi\phi\k,\dots,\psi\phi\k^4\},
\end{equation}
where $\psi(i):=3-2i$, i.e., $\psi=(1)(2\,4\,5\,3)$. It is a simple exercise that $\phi\psi=\psi\phi\k$, $\k\psi=\psi\k^2$ and $\psi^2=\phi\k^{4}$. Using these relations, one can check that triple products and inverses of elements of $\psi D_5$ are again in $\psi D_5$.

It turns out that $N$-ary subgroups admit a simple construction in terms of classical group theory. It constitutes the `easy' direction of the Post coset theorem \cite{GBV,Post,Shah}.
\begin{theorem}[Post]\label{PostConv} Let $H\subset\H\subseteq G$ be a pair of (binary) subgroups of $G$ with $H$ normal in $\H$ and $\H/H\simeq\Z_{N-1}$. Then every coset of $H$ in $\H$ is an $N$-ary subgroup of $G$.
\end{theorem}
\begin{proof} Let $C=bH$ for $b\in\H$. Since $H$ is normal in $\H$, for any $h\in H$ and $b\in\H$ we have $hb = bh'$ for some $h'\in H$. And since $\H/H$ is cyclic of order $N-1$, we have $b^{N-1}\in H$ for any $b\in\H$. Therefore, for $a_i=bh_i$ with $h_i\in H$, by induction,
$$
a_1\cdots a_N=bh_1 bh_2\cdots bh_N=b(bh_1'h_2b\cdots bh_N)=\dots=b(b^{N-1}\hat{h})\in bH=C.
$$
Similarly, since $(b^{-1})^{N-1}\in H$, moving $b^{-1}$-s to the left we obtain for $a=bh$, 
$$
a^{2-N}=(h^{-1}b^{-1})^{N-2}=(b^{-1})^{N-2}\hat{h}=b\big((b^{-1})^{N-1}\hat{h}\big)\in bH=C.
$$
Thus, $C$ is an $N$-ary subgroup.
\end{proof}
\noindent One can check that for the purely imaginary numbers we have $G=\+{\C}$, $H=\+{\R}$, and $\H=\langle i,\+{\R}\rangle=\+{\R}\cup i\+{\R}$. For the Thompson's ternary group, $G=S_5$, $H=D_5$ and $\H=\langle\psi,\phi,\k\rangle$. We will give an explicit description of $\H$ in Section \ref{Affine}. 

The hard direction of the Post coset theorem states that any $N$-ary group arises as a Post coset up to isomorphism, i.e. it is isomorphic to a coset of an index $N-1$ normal subgroup of a larger group with the cyclic quotient. It was proved by Post in his 1940 seminal paper on $N$-ary groups \cite{Post}. We will not elaborate further as we will not need it, but it does justify the following terminology \cite{GBV,Shah}.
\begin{definition}\label{PostCover}
Given an $N$-ary subgroup $C\subseteq G$, the group $H=C^{N-1}$ will be called its Post associate and the group $\H=\langle C\rangle$ its Post cover.
\end{definition} 
\noindent Post himself called them associated and covering group, respectively, without attaching his name. 

In general, the arity $N$ given by the Post coset theorem is not the smallest possible for a given coset as a set with multiplication. Indeed, if $C^4=H$ then $(C^2)^2=H$, so $C^2$ will not just be quinternary like $C$, but also ternary. Conversely, if $C$ is $N'$-ary then we can take the product of its $N'$ elements, multiply it by $N'-1$ more elements, then by $N'-1$ more, and so on. By the closure assumption, all of those products will again be in $C$. This means that closure under $N'$-ary multiplication implies closure under $N$-ary multiplication for any $N=N'+k(N'-1)$. In particular, closure under ternary products implies closure under any odd-numbered products. In the light of this ambiguity, we would like to distinguish the smallest possible arity for each coset.
\begin{definition} A subset $C$ of a group will be called strictly $N$-ary, and $N$ its strict arity, when it is closed under $N$-ary multiplication, but not under $N'$-ary multiplication with any $N'<N$. 
\end{definition}
\noindent It follows from the Post coset theorem that the strict arity is $N=\ord(C)+1$, where $\ord(C)$ is the order of $C$ as an element of the quotient group $\H/H$. Indeed, it is the smallest $i$ such that $C^i=C$.

\section{Group of affine permutations $\boldsymbol{\Aff(\Z_n)}$}\label{Affine}

In this section, we introduce the group whose subgroups will serve as Post covers for panmagic $N$-ary subgroups. Its elements are permutations realized by affine transformations of $\Z_n$ analogous to affine transformations of $\R$. 
An {\it affine transformation} of $\Z_n$ is a map of the form $\pi_{a,b}(i)\equiv ai+b\pmod{n}$ with some parameters $a,b\in\Z_n$, where $a$ is said to be the {\it slope} of the transformation. Clearly, $\pi_{a,b}$ is invertible, and hence a permutation of $\Z_n$, if and only if $a$ is invertible in $\Z_n$. It is a standard fact of number theory that this is happens if and only if $a$ and $n$ are relatively prime \cite{Rosen}.

\begin{definition}\label{pia} A permutation in $S_n$ is called affine when it is given by an invertible affine transformation of $\Z_n$. The group of affine permutations, also known as the general affine group of degree $1$ over $\Z_n$, is denoted $\Aff(\Z_n)$. 
\end{definition}
\noindent Aside from modular queens, affine permutations come up in enumerative combinatorics \cite{BK-V}, pseudo-random number generation \cite{HD,Mars}, and in pure number theory, for example, in Zolotarev's lemma from a proof of quadratic reciprocity. As with the dihedral group $D_n$, we will not distinguish between $\Aff(\Z_n)$ and its isomorphic representation by permutations, i.e. we will also treat it as a subgroup of $S_n$.  Let us denote $\+{\Z}_n$ the group of invertible residues (units) of $\Z_n$, then
\begin{equation}\label{GAmod}
\Aff(\Z_n)=\{\s(i)=ai+b\mid a\in\+{\Z}_n,\,b\in\Z_n\}.
\end{equation}
We leave it to the reader to show that $\Aff(\Z_n)$ can also be represented by $2\times2$ upper triangular matrices of the form $\begin{pmatrix}
a  &  b\\
0  &  1\end{pmatrix}$ with the usual matrix multiplication. 

We abbreviate $\pi_{a,0}$ as $\pi_a$, and these permutations will serve as convenient coset representatives since their powers are easy to express, $\pi_a^k=\pi_{a^k}$. Every affine permutation is a product of $\pi_a$ and a power of $\k$, indeed, $\pi_{a,b}=\k^b\pi_a$. Therefore, $\Aff(\Z_n)=\langle\pi_a,D_n\mid a\in\+{\Z}_n\rangle$. The commutation relations for $\pi_a$ are also easy to derive, they are $\pi_a\phi=\k^{a-1}\phi\pi_a$ and $\pi_a\k=\k^a\pi_a$. Thompson's permutation in \eqref{psiD5} is $\psi=\pi_{-2,3}=\k^3\pi_{-2}$, and $\psi D_5=\pi_{-2}D_5$. Since $-2$ is a primitive root modulo $5$ it generates all of $\+{\Z}_5$, and hence $\pi_{-2}$ generates $\pi_a$ for all units $a$. Therefore, the Post cover of $\psi D_5$ is $\langle\psi,\phi,\k\rangle=\langle\pi_{-2},D_5\rangle=\Aff(\Z_5)$.

A simple consequence of Lionnet's combinatorial characterization of panmagic permutations is the following.
\begin{corollary}\label{AfPmag} A permutation $\pi_{a,b}\in \Aff(\Z_n)$ is dihedral if and only if $a=\pm1$, and is panmagic if and only if $a,a\pm1\in\+{\Z}_n$. It is magic if and only if it is panmagic.
\end{corollary}
\noindent Indeed, $i\mapsto\pi_{a,b}(i)\pm i$ are permutations if and only if their slopes, $a\pm1$, are invertible. The last claim follows from the fact that an affine transformation is injective if and only if it is injective at a single point. Since these properties of affine permutations entirely depend on their slopes $a$ we will also call units of $\Z_n$ dihedral or panmagic accordingly. 

It is easy to see now why panmagic permutations exist in all Polya dimensions. Permutations $\pi_2$ and $\pi_3$ are panmagic whenever $2,3\nmid n$ because $2,2\pm1$ and $3,3\pm1$ are units for any such $n$. Moreover, since at least one of $a-1$, $a$, and $a+1$ is divisible by $2$ and at least one by $3$, we see directly that affine panmagic permutations do not exist when $n$ is even or divisible by $3$. The Polya theorem proves a stronger claim that even non-affine panmagic permutations do not exist in those dimensions. 

The number of units in $\Z_n$ is given by Euler's totient function $\pfi(n):=|\+{\Z}_n|$. When $n=p_1^{\a_1}\cdots p_m^{\a_m}$ is the prime factorization of $n$, we have \cite{Rosen}
\begin{equation}\label{totient}
\pfi(n)=n\left(1-\frac1{p_1}\right)\cdots\left(1-\frac1{p_m}\right).
\end{equation}
We now see from \eqref{GAmod} that $|\Aff(\Z_n)|=\pfi(n)n$. The number of panmagic units of $\Z_n$ can be similarly counted. By Corollary \ref{AfPmag}, they are in $1$-$1$ correspondence with triples of consecutive units, and Schemmel generalized Euler's totient to $k$-totient $\pfi_k(n)$ that counts $k$-tuples of consecutive units, see \cite{Lehm}. It can be calculated analogously to  \eqref{totient}: 
\begin{equation}\label{ktotient}
\pfi_k(n)=n\left(1-\frac{k}{p_1}\right)\cdots\left(1-\frac{k}{p_m}\right)
\end{equation}
when all $p_i\geq k$, and $0$ otherwise. The number of panmagic units is thereby $\pfi_3(n)$, and we see again that there are none in non-Polya dimensions. 

Are all panmagic permutations affine? This question received much attention in the modular $n$-queens literature, and the answer, in general, is negative. They are for $n=5,7,11$, but in all Polya dimensions $n\geq13$ there exist non-affine panmagic permutations. First examples of them were constructed only in 1975 by Bruen and Dixon, and their nature and properties are still poorly understood \cite{BS2}. In this paper, we will restrict our attention to affine panmagic permutations only.  

\section{Panmagic-preserving subgroups}\label{PMagPres}

We plan to apply the Post coset theorem to construct $N$-ary subgroups of affine panmagic permutations in all Polya dimensions. The bigger they are the more interesting they are, and, being cosets, they have the size equal to the size of their Post associates. Moreover, for all of their elements to be panmagic the elements of the Post associate must preserve panmagic when multiplied by them. Thus, we would like to find the largest subgroup of $S_n$ whose elements preserve panmagic when multiplied by affine panmagic permutations. Since panmagic of $\pi_{a,b}$ depends only on its slope $a$, this means that we are looking for the largest subgroup of $\+{\Z}_n$ whose elements when multiplied by a panmagic unit return another panmagic unit. 

By Corollary \ref{AfPmag}, a unit $a$ is panmagic if and only if $a\pm1$ are also units, or, equivalently, if and only if $a^2-1=(a-1)(a+1)$ is a unit. But $x$ is a unit modulo $n$ if and only if it shares no prime factors with $n$, so the panmagic condition reduces to $a^2\not\equiv1\pmod{p}$ for every prime $p\mid n$. This suggests applying the Chinese Remainder theorem (CRT) \cite{IrRos,Rosen} to characterize preservation of panmagic. 
\begin{theorem}\label{PmagStab} Let $u\in\+{\Z}_n$ with $2,3\nmid n$. Then $ua$ is panmagic for every panmagic $a\in\+{\Z}_n$ if and only if $u^2\equiv1\pmod{p}$ for every prime $p\mid n$.
\end{theorem} 
\begin{proof} If $u^2\equiv1\pmod{p}$ for every prime $p\mid n$ then the characteristic property of panmagic units, $a^2\not\equiv1\pmod{p}$, is trivially retained by $ua$. On the other hand, if $u^2\not\equiv1\pmod{p_i}$ for $i$ in some non-empty subset $I$ that indexes   prime factors of $n$ then set $a_i\equiv u^{-1}\pmod{p_i^{\a_i}}$ for $i\in I$ and $a_i\equiv2\pmod{p_i^{\a_i}}$ for $i\not\in I$.  Let $a\in\+{\Z}_n$ be the solution to $x\equiv a_i\pmod{p_i^{\a_i}}$ given by CRT. By construction, $a^2\not\equiv1\pmod{p_i}$ for all $i$ since $(u^{-1})^2\equiv(u^2)^{-1}\not\equiv1\pmod{p_i^{\a_i}}$ for $i\in I$ and $2^2\not\equiv1\pmod{p_i^{\a_i}}$ for any $i$ since $3\nmid n$, so $a$ is panmagic. But $ua\equiv ua_i\equiv1\pmod{p_i^{\a_i}}$ for $i\in I$, and, all the more, $ua\equiv1\pmod{p_i}$, so $ua$ is not panmagic. Thus, $u$ does not preserve panmagic.
\end{proof}
\noindent The panmagic preserving condition of Theorem \ref{PmagStab} can be neatly packaged by introducing the radical \cite{Nath}. For a positive integer $n$ with the prime factorization $n=p_1^{\a_1}\dots p_m^{\a_m}$, its {\it radical} (or square-free kernel) is the number $\rad(n):=p_1\cdots p_m$. An integer $n$ is called {\it square-free} when it is not divisible by any perfect square, i.e. when $\rad(n)=n$. By CRT, $x\equiv1\pmod{p}$ for every prime $p\mid n$ if and only if $x\equiv1\pmod[\big]{\!\rad(n)}$, and we get the following corollary.
\begin{corollary}\label{PanRad} Let $u\in\+{\Z}_n$ with $2,3\nmid n$. Then $u$ preserves panmagic, i.e. $ua$ is panmagic for every panmagic $a\in\+{\Z}_n$, if and only if $u^2\equiv1\pmod[\big]{\!\rad(n)}$.
\end{corollary}
\noindent In other words, panmagic-preserving units are closely related to modular square roots of unity and there can be many more of them than just $\pm1$.
\begin{definition}
We denote $\sqrt{1}_{\,n}$ the group of square roots of unity modulo $n$, and $\sqrt{1}_{n}^{\,\#}$ the group of lifted square roots of unity modulo $n$:
\begin{align*}
\sqrt{1}_{n}&:=\{a\in\Z_n^{\times}\mid a^2\equiv1\!\!\!\!\pmod{n}\}\\
\sqrt{1}_{\,n}^{\,\#}&:=\big\{a\in\Z_n^{\times}\mid a^2\equiv1\!\!\!\!\pmod[\big]{\!\rad(n)}\big\}.
\end{align*}
\end{definition}
\noindent The ``lifted" in the name comes from thinking of elements of $\sqrt{1}_{\,n}^{\,\#}$ as pre-images or `lifts' of elements of $\sqrt{1}_{\rad(n)}$ under the map $\Z_n\to\Z_{\rad(n)}$ that reduces residues modulo $\rad(n)$. 

In general, $\{\pm1\}\subseteq\sqrt{1}_{n}\subseteq\sqrt{1}_{\,n}^{\,\#}$ and both inclusions are proper. The first one becomes equality for prime powers, and the second one for square-free numbers. All three are equal only for primes. The group $\sqrt{1}_{n}$ is studied in \cite{OOO}. For odd $n$, it has a pair of elements per each distinct prime factor of $n$ as one can see from CRT. If we use the standard notation $\omega(n)$ for the number of distinct prime factors of $n$ \cite{Rosen} then $|\sqrt{1}_{n}|=2^{\omega(n)}$. And since each residue from $\Z_{\rad(n)}$ has $\frac{n}{\rad(n)}$ lifts in $\Z_n$ we get $\big|\sqrt{1}_{\,n}^{\,\#}\big|=\frac{n}{\rad(n)}\,2^{\omega(n)}$. The gain in size is substantial already for prime powers: $\big|\sqrt{1}_{\,p^\a}^{\,\#}\big|=p^{\a-1}|\sqrt{1}_{p^\a}|$. For example, $\sqrt{1}_{25}=\{1,24\}$, while $\sqrt{1}_{\,25}^{\,\#}=\{1,4,6,9,11,14,16,19,21,24\}$.

The dihedral group $D_n$ can be described as the subgroup of affine permutations with slopes $\pm1$. When we allow square roots and lifted square roots of unity as slopes instead, we get the extensions of it defined below. 
\begin{definition}
The multidihedral and the lifted dihedral groups are defined as
\begin{align*}
\mD_n&:=\{\sigma\in \Aff(\Z_n)\mid \s(i)=ui+b, u\in\sqrt{1}_{n},\,b\in\Z_n\}\\
\sD_n&:=\big\{\sigma\in \Aff(\Z_n)\mid \s(i)=ui+b, u\in\sqrt{1}_{n}^{\,\#},\,b\in\Z_n\big\}.
\end{align*}
\end{definition}
\noindent There is no standard notation or name for $\mD_n$, but  ``multidihedral" seems fitting (``multidi" for multiple pairs of roots) and is not taken.

We will also refer to elements of $\sqrt{1}_{n}$ and $\sqrt{1}_{\,n}^{\,\#}$ as multidihedral and lifted dihedral units, respectively. For prime powers, we have the {\it lifted dihedral/panmagic dichotomy}: every unit or affine permutation is either lifted dihedral or panmagic. Indeed, as there is a single prime, either $a^2\equiv1\pmod{p}$, and $a$ is lifted dihedral, or not, and $a$ is panmagic. For primes, ``lifted" can be dropped since $\sqrt{1}_{p}^{\,\#}=\{\pm1\}$ and $\sD_p=D_p$. But for composite $n$ with distinct prime factors, there are intermediate $a\in\+{\Z}_n$ with $a^2\equiv1\pmod{p}$ for some, but not all, prime factors $p$ of $n$.

Clearly, $|\mD_n|=|\sqrt{1}_{n}|\,n=2^{\omega(n)}n$ and  
$$
|\sD_n|=\big|\sqrt{1}_{n}^{\,\#}\big|\,n=\frac{n}{\rad(n)}\,2^{\omega(n)}n=\frac{n}{\rad(n)}\,|\mD_n|.
$$
In particular, $\sD_n$ can be much larger than $\mD_n$, and so will be its cosets. 

The lifted dihedral group $\sD_n$ is the largest subgroup of $S_n$ that preserves affine panmagic of permutations. Hence, it is the group that we want as the Post associate. But the simpler and smaller $\mD_n$ will also prove useful when considering panmagic cycle types.

\section{$\boldsymbol{\Aff(\Z_n)}$ as the Post cover}

For the Post coset theorem to apply with $\mD_n$ or $\sD_n$ as the Post associate and $\Aff(\Z_n)$ as the Post cover, we need them to be normal in $\Aff(\Z_n)$ and the quotient groups to be cyclic. In this section, we establish when this is the case and construct the corresponding $N$-ary subgroups.

\subsection{Slope-squaring homomorphisms}\label{SlopeSq}

A simple way to prove normality of subgroups and cyclicity of quotients by them is to use the homomorphism theorem \cite{Gallian}. If $f:G\to\Gamma$ is a homomorphism between two groups, with the kernel $\Ker f$ and the range $\Ran f$, then $\Ker f$ is a normal subgroup of $G$ and $G/\Ker f\simeq \Ran f$.

The definitions of $\mD_n$ and $\sD_n$ suggest what our homomorphisms should be. Indeed, consider the maps  $f(\pi_{a,b}):=a^2$ and $f_{\rad}(\pi_{a,b}):=a^2\!\pmod[\big]{\!\rad(n)}$. Squaring is a homomorphism from $\Z_n^{\times}$ to itself for any $n$, and so is the reduction $\+{\Z}_n\to\+{\Z}_{\rad(n)}$. That the slope-taking map is also a homomorphism is easy to see using the matrix representation of $\Aff(\Z_n)$: 
$$
\begin{pmatrix}
    a & b\\
    0 & 1\end{pmatrix}
    \begin{pmatrix}
        a' & b'\\
        0 & 1
    \end{pmatrix}=
    \begin{pmatrix}
            aa' & ab'+b\\
            0 & 1
        \end{pmatrix},
$$ 
so the slope of the product is the product of the slopes. Thus, both $f$ and $f_{\rad}$ are homomorphisms and, by inspection, their kernels are  $\mD_n$ and $\sD_n$, respectively. 
\begin{definition}
The map $f:\Aff(\Z_n)\to \Z_n^{\times}$ given by $f(\pi_{a,b}):=a^2$ will be called the slope-squaring homomorphism. Its range, the group of quadratic residues modulo $n$, will be denoted $QR_n:=\{a^2\mid a\in\Z_n^{\times}\}$. The map $f_{\rad}(\pi_{a,b}):=a^2\!\pmod[\big]{\!\rad(n)}$ will be called the reduced slope-squaring homomorphism.  \end{definition}
\noindent By the above definitions, $\Ker f=\mD_n$, $\Ran f=QR_n$, $\Ker f_{\rad}=\sD_n$ and $\Ran f_{\rad}=QR_{\rad(n)}$. Applying the homomorphism theorem, we obtain the following.
\begin{corollary} $\mD_n$ and $\sD_n$ are normal subgroups of $\Aff(\Z_n)$ and
\begin{align}\label{GAhomomorphism}
\Aff(\Z_n)\big/\,\mD_n&\simeq\,\Z_n^{\times}\big/\sqrt{1}_{n}\simeq QR_n\,;\notag\\
\Aff(\Z_n)\big/\,\sD_n&\simeq\,\Z_n^{\times}\big/\sqrt{1}_{n}^{\,\#}\simeq\Z_{\rad(n)}^{\times}\big/\sqrt{1}_{\rad(n)}\simeq QR_{\rad(n)}.
\end{align}
\end{corollary}
\noindent From the first isomorphism, $|QR_n|=\frac{\pfi(n)}{2^{\omega(n)}}$, and from the second, $|QR_{\rad(n)}|=\frac{\rad(n)}{n}\frac{\pfi(n)}{2^{\omega(n)}}$. For primes, $QR_p\simeq\Z_{\frac{p-1}{2}}$ because quadratic residues are exactly the even powers of a primitive root. For prime powers, $QR_{\rad(p^\a)}=QR_p$ and \eqref{GAhomomorphism} reveals large panmagic $N$-ary subgroups already at the level of units. As an amusing illustration, the ten panmagic units of $\+{\Z}_{25}$ form its ternary subgroup, i.e., $$2\,\sqrt{1}_{25}^{\,\#}=\{2,3,7,8,12,13,17,18,22,23\}$$ is closed under ternary multiplication and taking modular inverses. 

The remaining obstruction to applying the Post coset theorem is that for composite $n$ that are not prime powers $QR_n$ may not be cyclic. When it is not, $\Aff(\Z_n)$ cannot be the Post cover of $N$-ary cosets for any $N$. The simplest way out is to restrict to $n$ with $QR_n$ or $QR_{\rad(n)}$ cyclic, and we follow it in the next section. A more refined approach appears in Section \ref{StrictNary}.

\subsection{Square-cyclic numbers}\label{SqCyc}

While the absence of primitive roots in non-prime power Polya dimensions means that $\+{\Z}_{n}$ is not cyclic, this is not always the case for $QR_n$.
\begin{definition} We will call a positive integer $n$ square-cyclic when $QR_n$ is cyclic, and call $r\in\+{\Z}_n$ a square-primitive root of $n$ when $r^2$ generates $QR_n$.
\end{definition}
Which integers are square-cyclic? Certainly, odd primes and their powers are because they have  primitive roots and squares of primitive roots generate $QR_n$. To give a general answer, we need to understand the structure of $QR_n$. Recall that given the prime factorization of $n={p_1}^{\a_1}p_2^{\a_2}\dots p_m^{\a_m}$, we get the primary decomposition 
\begin{equation}\label{Zndecomp}
\+{\Z}_n\simeq\+{\Z}_{{p_1}^{\!\a_1}}\times\cdots\times\+{\Z}_{{p_m}^{\!\!\a_m}},
\end{equation}
with the isomorphism given by taking $x\in\+{\Z}_n$ modulo each ${p_i}^{\!\a_i}$ \cite{IrRos,Nath}. The next result is folklore, but goes back, essentially, to Gauss. Recall that $\pfi(n)$ denotes Euler's totient. 
\begin{lemma}\label{QRndec} Let $n=p_1^{\a_1}p_2^{\a_2}\dots p_m^{\a_m}$ be the prime factorization of an odd integer $n$. Then 
\begin{equation}\label{QRndecomp} 
QR_n\simeq QR_{p_1^{\!\a_1}}\!\times\cdots\times QR_{{p_m}^{\!\!\!\a_m}},
\end{equation}
and $QR_n$ is cyclic if and only if $\frac{\pfi(p_i^{\a_i})}2$ are pairwise relatively prime. $QR_{\rad(n)}$ is cyclic if and only if $\frac{p_i-1}2$ are pairwise relatively prime.
\end{lemma}
\begin{proof} Restrict the isomorphism in \eqref{Zndecomp} to the quadratic residues $QR_n$. It is still injective and its range is contained in $QR_{{p_1}^{\!\a_1}}\!\times\cdots\times QR_{{p_m}^{\!\!\a_m}}$ because if $x\equiv a^2\pmod{n}$ then $x\equiv a^2\pmod{p_i^{\a_i}}$. It is also surjective by the Chinese Remainder theorem. Indeed, let $a_i^2\in\+{\Z}_{{p_i}^{\a_i}}$ be any quadratic residue for every $i=1,\dots,m$. Then the system of congruences $x\equiv a_i\pmod{p_i^{\a_i}}$ has a (unique) solution $x=a$ in $\+{\Z}_n$, and its square $a^2$ is a quadratic residue that gets mapped to $a_i^2$. Thus, $QR_n$ is isomorphic to the direct product of $QR_{{p_i}^{\a_i}}$. 

Since each $\Z_{{p_i}^{\a_i}}^{\times}$ is cyclic of order $\pfi({p_i}^{\a_i})$ its subgroup $QR_{{p_i}^{\a_i}}$ is cyclic of order $\frac{\varphi({p_i}^{\a_i})}{2}$ because half of the powers of a primitive root are even. It remains to note that a direct product of cyclic groups is cyclic if and only if the orders of its factors are pairwise relatively prime \cite{Gallian}.  
\end{proof}
It follows that there are plenty of square-cyclic numbers that are not prime powers, for example, $35=5\cdot7$, $275=5^2\cdot11$ or $385=5\cdot7\cdot11$. Moreover, even when $n$ does not generate a cyclic $QR_n$, its quotient $QR_{\rad(n)}$ can still be cyclic. Recall that $|QR_{\rad(n)}|=\frac{\rad(n)}{n}\frac{\pfi(n)}{2^{\omega(n)}}$, where $\omega(n)$ is the number of distinct prime factors of $n$. Our first result on panmagic $N$-ary groups concerns Polya dimensions with square-cyclic radicals. 
\begin{theorem}\label{comptheory} Let $n\geq5$ have a square-cyclic radical, $2,3\nmid n$, $r$ be a square-primitive root of $\rad(n)$, $N:=\frac{\rad(n)}{n}\frac{\pfi(n)}{2^{\omega(n)}}+1$ and $C:=\pi_r\sD_{n}$. Then $C^i=\pi_r^i\,\sD_{n}$  are $N$-ary subgroups consisting of panmagic permutations when $r^i$ is panmagic. Moreover, 
$$
\langle\pi_r,\sD_n\rangle=\cup_{i=1}^{N-1}C^i=\Aff(\Z_n),
$$
$\sD_n$ is its normal subgroup, and the quotient group is $\Aff(\Z_n)/\,\sD_n\simeq\Z_{N-1}$. 
\end{theorem}
\noindent The proof amounts to a direct application of the Post coset theorem to \eqref{GAhomomorphism}, so we leave details to the reader. Theorem \ref{comptheory} does not claim that $C^i$ are panmagic for all $1\leq i\leq N-2$, nor does it say explicitly which of them are. For prime powers, all $C^i$, other than $\sD_n$ itself, are panmagic due to the lifted dihedral/panmagic dichotomy. In general, we have to check each $i$ individually, but $C=\pi_r\sD_{n}$ itself, at least, is always panmagic. This is because if $a^i$ is panmagic for some $i$ then so is $a$ itself. Indeed, 
\begin{equation}\label{a^2k-1}
(a^i)^2-1=(a^2)^i-1=(a^2-1)(a^{2(i-1)}+\cdots+a^2+1),
\end{equation}
so if $(a^i)^2-1$ is invertible then so is $a^2-1$. If $r$ is a square-primitive root then the powers of $\pi_r\sD_{n}$ exhaust all cosets of $\sD_{n}$. Some of them must be panmagic when $n$ is Polya. As all of them are of the form $\pi_{r^i}\sD_{n}$, some powers $r^i$ must be panmagic, and hence $r$ is itself panmagic. 

\section{Strictly $N$-ary panmagic subgroups}\label{StrictNary}

Our first attempt to apply the Post coset theorem to construction of panmagic $N$-ary subgroups in Theorem \ref{comptheory} is somewhat unsatisfactory. It does not apply to all $n$ and does not distinguish panmagic and non-panmagic cosets. Upon reflection, the root of the problem is in insisting on $\Aff(\Z_n)$ as the Post cover. In this section, we pursue an alternative, adjusting the Post cover individually to each coset. Aside from fixing the above two problems, it will give us strict arities of individual $N$-ary subgroups rather than an umbrella value for all of them. 

When the coset is $\pi_a\sD_n$ and $a$ is panmagic, the trick is to restrict the reduced slope-squaring homomorphism from all of $\Aff(\Z_n)$ to its subgroup $\langle\pi_a,\sD_n\rangle$. It can always serve as the Post cover because its quotient always happens to be cyclic, unlike $QR_{\rad(n)}$. Analogous result also holds for generally smaller cosets of $\mD_n$. Recall that the {\it multiplicative order} $\ord_n(x)$ of $x\in\Z_n^{\times}$ is the smallest positive integer $i$ such that $x^i\equiv1\pmod{n}$.
\begin{theorem}\label{StrictLDnarity}\label{compstrictarity} Let  $a\in\Z_n^{\times}$ be panmagic, and $N:=\ord_{\rad(n)}(a^2)+1$ ($N:=\ord_{n}(a^2)+1$). Then $C:=\pi_a\sD_n$ ($C:=\pi_a\mD_n$) is a strictly $N$-ary subgroup of $S_n$ consisting of panmagic permutations with the Post associate $\sD_n$ ($\mD_n$) and the Post cover $\langle\pi_a,\sD_n\rangle$ ($\langle\pi_a,\mD_n\rangle$). 
\end{theorem} 
\begin{proof}
We give the proof for $\pi_a\sD_n$, the proof for $\pi_a\mD_n$ is analogous. Let $f_{\rad}:\langle\pi_a,\sD_n\rangle\to\Z_{\rad(n)}^{\times}$ be the reduced slope-squaring homomorphism restricted to $\langle\pi_a,\sD_n\rangle$. We still have $\Ker f_{\rad}=\sD_n$, but $\Ran f$ now only consists of squares of slopes of permutations from $\langle\pi_a,\sD_n\rangle$. Since $\langle \pi_a,\sD_n\rangle$ is the set of products of powers of $\pi_a$ and elements of $\sD_n$, their slopes are of the form $\e a^i$ for $\e\in\sqrt{1}_{\,n}^{\,\#}$. Therefore, their squares are $(\e a^i)^2=\e^2(a^2)^i\equiv(a^2)^i\!\!\pmod[\big]{\!\rad(n)}$ because $\e^2\equiv1\!\!\pmod[\big]{\!\rad(n)}$ by definition of $\sqrt{1}_{\,n}^{\,\#}$. Thus, $\Ran f_{\rad}\subseteq\langle a^2\rangle_{\Z_{\rad(n)}^{\times}}$\!, and we have equality because $f_{\rad}(\pi_a^i)\equiv(a^2)^i\!\!\pmod[\big]{\!\rad(n)}$.

The group $\langle a^2\rangle_{\Z_{\rad(n)}^{\times}}$\! is cyclic of size $\ord_{\rad(n)}(a^2)=N-1$, so most of the claim follows from the homomorphism theorem and Theorem \ref{PostConv}. That the arity $N$ is strict follows from the fact that $\ord_{\rad(n)}(a^2)$ is the smallest $i$ such that $\pi_a^i\in\sD_n$, so $\pi_a^{i+1}\not\in\pi_a\sD_n$ for positive $i+1<N$. 
\end{proof}
We can use Theorem \ref{compstrictarity} to find dimensions $n$ that admit strictly $N$-ary panmagic cosets of $\sD_n$. The ternary case $N=3$ is relatively straightforward.
\begin{corollary}\label{PanTerComp}
Let $2,3\nmid n$. Then, strictly ternary affine panmagic cosets of $\sD_n$ exist if and only if $-1$ is a quadratic residue modulo $\rad(n)$. Moreover, if $a^2=-1$ then $\pi_a\sD_n$ is such a ternary coset. Equivalently, such cosets exist if and only if all prime factors of $n$ are of the form $4k+1$ with some integers $k$. 
\end{corollary}
\begin{proof}
If $\pi_a\sD_n$ is strictly ternary then, by Theorem \ref{compstrictarity},
$\ord_{\rad(n)}(a^2)=2$, i.e. $a^4=1$. But then $a^4-1=(a^2-1)(a^2+1)=0$. Since $a$ is panmagic $a^2-1$ is a unit, so $a^2+1=0$ and $-1$ is a quadratic residue modulo $n$. Conversely, if $-1=a^2$ then $a^2-1=-2$ is a unit since $n$ is odd, so $a$ is panmagic. Moreover, $a^4-1=0$ and $2$ is the smallest exponent such that $(a^2)^2=1$, i.e. $\ord_{\rad(n)}(a^2)=2$. Thus, by Theorem \ref{compstrictarity}, $\pi_a\sD_n$ is strictly $N$-ary with $N=\ord_{\rad(n)}(a^2)+1=3$.

That $-1$ is a quadratic residue modulo $n$ or $\rad(n)$ if and only if all prime factors of $n$ are $4k+1$ primes is a well-known result of number theory \cite{IrRos} (and it follows from Lemma \ref{QRndec}).
\end{proof}
\noindent This is a far-reaching generalization of Thompson's observation that $\pi_{-2}D_5$ is a ternary subgroup. The smallest composite with such a ternary subgroup is $25$ with $\pi_2\sD_{25}$, and the smallest non-prime power is $65=5\cdot13$ with $\pi_8\sD_{65}$. Primes of the form $4k+1$, $5,13,17,29,37,41, ...$\,, are well-known in number theory since Girard and Fermat, who conjectured that they are the only primes representable as the sums of two perfect squares. This conjecture was later proved by Euler  \cite{IrRos}.

For characterizing groups of higher arities, there is no slick trick, and we need to better understand multiplicative orders of squares of panmagic units. By \eqref{QRndecomp}, $QR_{\rad(n)}\simeq QR_{p_1}\!\times\cdots\times QR_{{p_m}}$, so
\begin{equation*}
\ord_{\rad(n)}(a^2)=\lcm\big(\!\ord_{p_1}(a^2),\dots,\ord_{p_m}(a^2)\big)
\end{equation*}
For a prime $p$, all units $a\in\Z_p^{\times}$ are powers of a primitive root $r$, so all $a^2$ are powers of $r^2$. But $\ord_p(r^2)=\frac{p-1}{2}$, so the orders of its powers, $(r^2)^k$, are exactly the divisors of $\frac{p-1}{2}$. These are all and only orders of quadratic residues modulo a prime $p$. The panmagic condition excludes only order $1$ as $\pm1$ are the only non-panmagic units modulo a prime. Therefore, the set of orders of panmagic squares modulo $\rad(n)$ can be described as follows
\begin{equation}\label{Ord2rad}
\Big\{\lcm\big(d_1,\dots,d_m\big)\ \Big|\ d_i\,\big|\,\frac{p_i-1}2,\,d_i\geq2\Big\}.
\end{equation}
However, we can do better than that.
\begin{lemma}\label{PmagOrders} Let $2,3\nmid n=p_1^{\a_1}p_2^{\a_2}\dots p_m^{\a_m}$. Then $h=\ord_{\rad(n)}(a^2)$ for some panmagic unit $a\in\+{\Z}_n$ if and only if $h\mid\lcm\big(\frac{p_1-1}2,\dots,\frac{p_m-1}2\big)$ and $h$ is not relatively prime to any $\frac{p_i-1}2$. 
\end{lemma}
\begin{proof} Let $h$ be in the set \eqref{Ord2rad}, then 
$$
h=\lcm\big(d_1,\dots,d_m\big)\ \big| \lcm\Big(\frac{p_1-1}2,\dots,\frac{p_m-1}2\Big)
$$ because $d_i\,\big|\,\frac{p_i-1}2$ by definition, and 
$$
\gcd\Big(\lcm\big(d_1,\dots,d_m\big)\,,\frac{p_i-1}2\Big)\geq\gcd\Big(d_i\,,\frac{p_i-1}2\Big)=d_i\geq2,
$$
so $h$ is not relatively prime to any $p_i$.

Conversely, let $h$ satisfy the two conditions of the lemma, and set $d_i:=\gcd\big(h,\frac{p_i-1}2\big)$. By assumption, $d_i\geq2$, and since $\gcd$ distributes over $\lcm$ 
$$
\lcm\big(d_1,\dots,d_m\big)=\gcd\left(h,\lcm\Big(\frac{p_1-1}2,\dots,\frac{p_m-1}2\Big)\right)=h,
$$ 
so $h$ is in the set \eqref{Ord2rad}.
\end{proof}
The desired characterization of strict panmagic arities is now at hand.
\begin{theorem}\label{CompStrictArities} Let $2,3\nmid n=p_1^{\a_1}p_2^{\a_2}\dots p_m^{\a_m}$ and $N-1={q_1}^{\b_1}\dots{q_m}^{\b_m}$ be the prime factorizations. Then $N$ is a strict arity of some panmagic coset of $\sD_n$ if and only if $N-1\mid\lcm\big(\frac{p_1-1}2,\dots,\frac{p_m-1}2\big)$ and $N-1$ is not relatively prime to any $\frac{p_i-1}2$. Equivalently, strictly $N$-ary affine panmagic cosets of $\sD_n$ exist if and only if all $p_i$ are of the form $2q_jk+1$ with some $j$, and for every $j$ there exists $p_i$ of the form $2{q_j}^{\!\b_j}l+1$.
\end{theorem}
\begin{proof} 
The first formulation is a direct consequence of Theorem \ref{compstrictarity} and Lemma \ref{PmagOrders} applied to $h=N-1$. For the second, note that to get $\frac{{p_i}-1}{2}$ not relatively prime to $N-1$, we must have $p_i=2q_jk+1$ with some $j$. And to get $N-1\mid\lcm\big(\frac{p_1-1}{2},\dots,\frac{p_m-1}{2})$, we must have  ${q_j}^{\!\b_j}\mid \lcm(\frac{p_1-1}{2},\dots,\frac{p_m-1}{2})$ for every $j$. Therefore, for every $j$ some $p_i$ must be of the form $2{q_j}^{\!\b_j}l+1$.
\end{proof} 
\noindent When there is a single prime factor $p$, the first condition becomes $N-1\mid\frac{p-1}2$ and the second then reduces to $N-1\geq2$. Moreover, we see that strict arities for primes and their powers are exactly the same.
\begin{corollary}\label{primeallstrictarity}
Let $p\geq5$ be prime and $N\geq3$. Then $N$ is a strict arity of some panmagic coset of $\sD_{p^\a}$ if and only if $p=2(N-1)k+1$ for some integer $k$. 
\end{corollary}
\noindent For $N=3$ we reproduce a result from Corollary \ref{PanTerComp}, where $4k+1$ primes made their appearance.

In general, different $q_j\neq q_{j'}$ in Theorem \ref{CompStrictArities} can be covered by the same $p_i$ of the form $2{q_j}^{\b_j}{q_{j'}}^{\b_{j'}}k+1$. As an illustration, consider strictly $13$-ary panmagic cosets. Since $N-1=2^2\cdot 3$, all prime factors of $n$ must be of the form $4k+1$ or $6k+1$, and there must be at least one prime factor of the form $8k+1$ and at least one of the form  $6k+1$.  The latter can be distinct or they can both be the same factor of the form $24k+1$, so both $n=7\cdot17$ and $n=73$ work and hence admit strictly $13$-ary panmagic cosets.

\section{Cycle types of affine panmagic permutations}\label{CycAffPan}

\subsection{Cycle types and fixed points}

The cycle structure of affine permutations can be quite intricate in general \cite{Mars}, and its explicit description does not seem to be known. Luckily, panmagic permutations are not general in one important respect -- they always have exactly one fixed point. This simply rephrases a part of their definition that the trace of their matrix has to be $1$. On the other hand, the fixed point equation $\pi_{a,b}(x)=ax+b=x$ is equivalent to $(a-1)x=-b$, which has solutions if and only if 
$\gcd(a-1,n)\mid b$. The exact same condition comes up in a group-theoretic context, and is key to describing panmagic cycle types. 
\begin{lemma}\label{FixConj}
An affine permutation $\pi_{a,b}$ is conjugate to $\pi_a$ if and only if it has fixed points, i.e. $\gcd(a-1,n)\mid b$, and then it can be conjugated into $\pi_a$ by a power of $\k$. 
\end{lemma}
\begin{proof}
One direction is trivial since $\pi_a$ always has a fixed point, $0\equiv n$, and conjugate permutations have the same cycle type. For the converse, conjugating $\pi_a$ by $\k^t$ gives $\k^{-t}\pi_a\k^t(x)=ax+(a-1)t$. We can make it $\pi_{a,b}$ when $(a-1)t=b$ has a solution, i.e. $\gcd(a-1,n)\mid b$. But this happens if and only if $\pi_{a,b}$ has a fixed point.
\end{proof}
Recall that two permutations have the same cycle types if and only if they are conjugate in $S_n$. So Lemma \ref{FixConj} allows us to characterize cycle types of all affine permutations with fixed points if we can do it for $\pi_a$. The cycle type of $\pi_a$ is relatively easy to find because powers of $\pi_a$ are easy to track. The following description of cycle types follows from the one proved in \cite{BK-V} for $\pi_a$, where combinatorial applications to Polya enumeration with the symmetry group $\Z_n^{\times}$ are discussed.
\begin{theorem}\label{GAcycletypes} Let $a$ be the slope of an affine permutation with fixed points, in particular, of a panmagic permutation. Then it has $\frac{\varphi(d)}{\ord_d(a)}$ cycles of length $\ord_d(a)$ for every $d\mid n$, and no other cycles.  
\end{theorem}
\begin{proof} Consider a cycle containing some $z\in\Z_n$ with $\gcd(z,n)=d'$. Starting with $z$, we generate the cycle $(z\ az\ a^2z\ \dots\  a^{m-1}z)$. Let $m$ be such that $a^mz=z$ or $(a^m-1)z\equiv 0\pmod{n}$. Reducing the moduli, this congruence is equivalent to $a^m-1\equiv 0\pmod{n/d'}$. The length of the cycle is the minimal such $m$, which is, by definition, $\ord_{n/d'}(a)$. As $\gcd(z,n)=d'$ is equivalent to $\gcd(z/d',n/d')=1$, $z\mapsto z/d'$ gives a $1$-$1$ correspondence between $z$ in $\Z_n$ with $\gcd(z,n)=d'$ and units in $\Z_{n/d'}$. Since there are $\varphi(n/d')$ units in $\Z_{n/d'}$ and $\ord_{n/d'}(a)$ of them in each cycle, we count $\frac{\varphi(n/d')}{\ord_{n/d'}(a)}$ such cycles. Since $\gcd(d',n)=d'$ our $d'$ runs over all and only divisors of $n$ as $z$ runs over all residues in $\Z_n$, and then so does $d=n/d'$.
\end{proof}
Equipped with Theorem \ref{GAcycletypes}, we can distinguish panmagic from other affine permutations based on their cycle types. 
\begin{corollary}\label{panmagiccycletype}
An affine permutation is panmagic if and only if its disjoint cycle decomposition has exactly one fixed point and no transpositions. 
\end{corollary}
\begin{proof}
As we know from Section \ref{PMagPres}, the panmagic condition $a^2-1\in \Z_n^{\times}$ is equivalent to $a^2\not\equiv 1\pmod{d}$ for any $d\mid n$, $d\neq 1$. But this means that $\ord_d(a)\geq3$ for any $d\neq 1$, so for any $z\neq n$, the cycle containing $z$ will have length at least  $3$ and $n$ will be a fixed point. 
\end{proof}
A similar argument shows that $\pi_a$ is multidihedral, i.e. $a^2\equiv 1\pmod{n}$, if and only if all of its disjoint cycles are fixed points or transpositions.
Indeed, then $\ord_n(a)=2$ and for any $d\mid n$ we have $\ord_d(a)\leq 2$. However, multidihedral or even dihedral permutations, unlike panmagic ones, do not need to have fixed points. The cyclic shift $\k$ is a striking counterexample with a single cycle of full length $n$. Such ``full period" permutations, but with less predictable number pattern, are described by the Hull-Dobell theorem and produced by linear congruential pseudo-random generators \cite{HD,Mars}.

\subsection{Monotypic cosets of $\mD_n$}

To complete our characterization of affine panmagic cycle types, we will describe their cosets where all permutations have the same cycle type, {\it monotypic} cosets for short. As we will prove in the next section, lifted dihedral cosets are never monotypic in non-square-free dimensions. However, the situation is much more interesting for multidihedral cosets that are strictly smaller in those dimensions, and they are the ones we study in this section. We start by relating the multiplicative orders of $a$ and $a^2$.
\begin{lemma}\label{ord+-a}
If $\ord_n(a^2)$ is even then $\ord_n(a)=\ord_n(-a)=2\ord_n(a^2)$. If $\ord_n(a^2)$ is odd then one of $\ord_n(-a)$ and $\ord_n(a)$ is equal to $\ord_n(a^2)$ and the other to $2\ord_n(a^2)$. 
\end{lemma}
\begin{proof} 
By a standard result of group theory \cite{Gallian}, 
\begin{equation}\label{orda^2} 
\ord_n(a^2)=\begin{cases}\ord_n(a),\,\ord_n(a)\text{ odd}\\\frac12\ord_n(a),\,\ord_n(a)\text{ even.}\end{cases}
\end{equation}
Hence, if $\ord_n(a^2)$ is even then so is $\ord_n(a)$. Indeed, if it were odd then by \eqref{orda^2} $\ord_n( a)=\ord_n(a^2)$ would also be odd, a contradiction. But then, by \eqref{orda^2} again, $\ord_n( a)=2\ord_n(a^2)$, and the same argument applies to $-a$.

Now suppose $\ord_n(a^2)$ is odd. Since  $\ord_n(a^2)=\ord_n\big((-a)^2\big)$, formula \eqref{orda^2} implies that $\ord_n(\pm a)$ is either $\ord_n(a^2)$ or $2\ord_n(a^2)$. For definiteness, suppose $\ord_n(a)=\ord_n(a^2)$. Since it is odd $(-a)^{\ord_n(a^2)}=- a^{\ord_n(a)}=-1$, so $\ord_n(-a)\neq\ord_n(a^2)$ and it must be $2\ord_n(a^2)$.
\end{proof}
We can now characterize the longest cycles of panmagic permutations. The characterization is based on Lemma \ref{ord+-a} and the following simple observation: $\ord_{d'}(x)\mid\ord_d(x)$ when $d'\mid d$. Indeed, if 
$x^m\equiv 1\pmod{d}$ then all the more $x^m\equiv 1\pmod{d'}$, and $\ord_{d'}(x)$ must divide any such $m$ by Lagrange's theorem. 
\begin{corollary}\label{CompLongCyc} Let $2,3\nmid n\geq5$, $a\in\+{\Z}_n$ be panmagic and $N=\ord_n(a^2)+1$ be the strict arity of $\pi_a\mD_n$. When $N$ is odd, the longest cycles of all permutations in $\pi_a\mD_n$ are of length $2(N-1)$. When $N$ is even, the longest cycles are of length $2(N-1)$ in some of them and of length $N-1$ in others.
\end{corollary}
\begin{proof}
Slopes of permutations in $\pi_a\mD_n$ are $\e a$ with $\e\in\sqrt{1}_n$, so, by Theorem \ref{GAcycletypes}, their cycle lengths are given by  $\ord_d(\e a)$ with $d\mid n$. Since $(\e a)^2=a^2$, Lemma \ref{ord+-a} applied to $\e a$ implies that $\ord_d(\e a)=\ord_d(-\e a)=2\ord_d(a^2)$ when $\ord_d(a^2)$ is even. When it is odd and $d\neq1$, then one of them is $\ord_d(a^2)$ and the other is $2\ord_d(a^2)$. Both cases give that the longest cycle of $\pi_{\e a}$ has the length  $2\ord_n(a^2)=2(N-1)$ because every other $\ord_{d'}(n)\mid 2(N-1)$ and all permutations in $\pi_a\mD_n$ are conjugate to $\pi_{\e a}$ for some $\e\in\sqrt{1}_n$.
\end{proof}
As follows from Theorem \ref{GAcycletypes}, the only cycles in the decomposition of a panmagic permutation in prime dimensions are the longest cycles and the fixed point. Therefore, Corollary \ref{CompLongCyc} suffices to characterize monotypic cosets in the prime case. However, there can be many more cycles in composite dimensions. Of course, if $N$ is even then Corollary \ref{CompLongCyc} immediately rules out monotypism. Moreover, its proof suggests monotypism precisely when $\ord_d(a^2)$ is even for all $d\mid n$, $d\neq1$. Unfortunately, there is a complication. Of course, if $\ord_d(a^2)$ is odd then $\pi_{\e a}$ and $\pi_{-\e a}$ will have non-matching cycles for this $d$, but they may be matched by cycles for some other $d'\mid n$. To rule out this possibility for at least some $d$, we need to look into even and odd multiplicative orders modulo divisors of $n$ more closely.

The next lemma shows that among all divisors $d$ of $n$ with $\ord_d(x)$ odd, there is the greatest one by divisibility. To prove it, we will need two elementary properties of multiplicative orders. First, when $k$ and $l$ are relatively prime, $\ord_{kl}(x)=\lcm\!\big(\!\ord_{k}(x),\ord_{l}(x)\big)$ \cite{Rosen}. This is a consequence of the Chinese Remainder theorem. And second, $\ord_{p^{\a}}(x)=p^k\ord_{p}(x)$ with some $k\leq\a$ for any odd prime $p$ \cite{Nath}.
\begin{lemma}\label{ordaevenodd}
Let $n$ be odd. Then there exists $h\mid n$ such that for any $d\mid n$ with odd $\ord_{d}(x)$ we have $d\mid h$. Moreover, all odd $\ord_d(x)$ with $d\mid n$ are divisors of $\ord_h(x)$, and $\ord_d(x)$ is even for all $d\mid n$, $d\neq1$ if and only if $\ord_p(x)$ is even for every prime $p\mid n$.
\end{lemma}
\begin{proof} 
Let $n=p_1^{\a_1}\cdots\,p_m^{\a_m}$. Define $h$ to be the product of all $p_i^{\a_i}$ such that $\ord_{p_i}(x)$ is odd, or $h:=1$ when there are no such $p_i$. Now let $d\mid n$ with  $\ord_{d}(x)$ odd. Then $d=p_{i_1}^{\b_1}\cdots\,p_{i_s}^{\b_s}$, where $\b_j:=\a_{i_j}$, and 
$$
\ord_{d}(x)=\lcm\big(\!\ord_{p^{\b_1}_{i_1}}(x),\dots,\,\ord_{p_{i_s}^{\b_s}}(x)\big). 
$$ 

Since all $p\mid n$ are odd for $n$  odd and $\ord_{p^{\b}}(x)=p^k\ord_{p}(x)$, the lcm can only be odd when all $\ord_{p_{i_j}}(x)$ are odd. Therefore, $p_{i_j}$ are among the $p_i$ used to form $h$. Moreover, $\b_j\leq\a_{i_j}$ since $d\mid n$, so $d\mid h$ and $h$ has the claimed property. Then $\ord_d(x)\mid\ord_h(x)$ when $\ord_d(x)$ is odd because $d\mid h$. On the other hand, if $\ord_{p}(x)$ is even for all prime $p\mid n$ then $h=1$ and $\ord_d(x)$ is even for all $d\mid n$, $d\neq1$, and vice versa.
\end{proof}
We are now ready to characterize monotypic cosets in composite dimensions.
\begin{theorem}\label{componecycle} Let $2,3\nmid n\geq5$, $a\in\+{\Z}_n$ be panmagic and $N=\ord_n(a^2)+1$ be the strict arity of $\pi_a\mD_n$. Then all of its permutations have the same cycle type if and only if $\ord_{p}(a^2)$ is even for every prime $p\mid n$.
\end{theorem}
\begin{proof} 
Suppose that $\ord_{p}(a^2)$ is even for every prime $p\mid n$. Then, by Lemma \ref{ordaevenodd}, $\ord_{d}(a^2)$ is even for all $d\mid n$, $d\neq1$. But all elements of $\pi_a\mD_n$ are conjugate to some $\pi_{\e a}$ with $\e\in \sqrt{1}_n$ by Lemma \ref{FixConj}, and the cycle lengths in $\pi_{\e a}$ and their counts depend only on $d$ and $\ord_d(\e a)$ by Theorem \ref{GAcycletypes}. Since $\ord_d(a^2)$ is even for all $d\mid n$ and $d\neq1$, $\ord_d(\e a)=2\ord_d(a^2)$, and thus,  does not depend on $\e$. This is also true for $d=1$ since $\ord_1(x)=1$ for any $x$. Therefore, every $\pi_{\e a}$ with $\e\in\sqrt{1}_n$ has the same cycle type, and thus, every permutation in $\pi_a\mD_n$ will also have the same cycle type.

Conversely, suppose that $\ord_p(a^2)$ is odd for some prime $p\mid n$. Then, by Lemma \ref{ordaevenodd}, we can construct $h>1$ as the largest divisor of $n$ with $\ord_h(a^2)$ odd. We know that there exists some $\pi_{\e a}$ where $\ord_h(\e a)=2\ord_h(a^2)$ but $\ord_h(-\e a)=\ord_h(a^2)$. If $\pi_{-\e a}$ has the same cycle type as $\pi_{\e a}$ then there is $h'$ with $\ord_{h'}(-\e a)=2\ord_h(a^2)$. But then $\ord_{h'}(a^2)=\ord_h(a^2)$ is odd and $h'\mid h$ by the characteristic property of $h$, so 
$$
2\ord_h(a^2)=\ord_{h'}(-\e a)\,\,\big|\,\ord_{h}(-\e a)=\ord_h(a^2),
$$
a contradiction. Thus, the cycle types of $\pi_{\e a}$ and $\pi_{-\e a}$ are distinct.
\end{proof}
For prime powers, since $\ord_{p^\a}(a^2)$ has the same parity as $\ord_{p}(a^2)$ monotypism is determined by the strict arity $N=\ord_{p^\a}(a^2)+1$. Other cases when $N$ is enough to determine monotypism are when $N-1$ is odd or a power of $2$. But generally, whether $\pi_a\mD_n$ is monotypic cannot be told from $N$ alone.

As we prove next, monotypism imposes a restrictive condition on $n$. Quite unexpectedly, it is the same condition as for the existence of strictly ternary panmagic cosets of $\mD_n$ and $\sD_n$. 
\begin{theorem}
$\mD_n$ has monotypic panmagic cosets if and only if all prime factors of $n$ are of the form $4k+1$. Moreover, then $\mD_n$ also has strictly ternary panmagic cosets with all permutations in them having only $4$-cycles and a single fixed point. 
\end{theorem}
\begin{proof} 
By assumption, there exists a monotypic panmagic coset $\pi_a\mD_n$, and by Theorem \ref{componecycle}, then $\ord_p(a^2)$ must be even for every prime $p\mid n$. But $\ord_p(a)=2\ord_p(a^2)$ by Lemma \ref{ord+-a}, so $4\mid\ord_p(a)$, and $\ord_p(a)\mid p-1$ by Lagrange's theorem. Thus, $4\mid p-1$ and all prime factors of $n$ must be of the form $4k+1$. 

As we know from number theory, $-1$ is then a quadratic residue modulo $n$ and there exists $\tilde{a}$ with $\tilde{a}^2=-1$ \cite{IrRos}. This $\tilde{a}$ is also panmagic because $\tilde{a}^2-1=-2$ is a unit. Since 
$\ord_n(\tilde{a}^2)=\ord_n(-1)=2$ the strict arity of $\pi_{\tilde{a}}\mD_n$ is $N=\ord_n(\tilde{a}^2)+1=3$, and $\ord_n(\e\tilde{a})=4$ for all $\e\in\sqrt{1}_n$ by Lemma \ref{ord+-a}. As panmagic permutations have no transpositions by Corollary \ref{panmagiccycletype}, $\pi_{\e\tilde{a}}$ can only have $4$-cycles and a single fixed point. But every permutation in $\pi_{\tilde{a}}\mD_n$ is conjugate to one of $\pi_{\e\tilde{a}}$ and so has the same cycle type. Thus, $\pi_{\tilde{a}}\mD_n$ is a strictly ternary panmagic coset with all permutations having only $4$-cycles and a single fixed point.

Conversely, if all prime factors of $n$ are of the form $4k+1$ then we can construct $\pi_{\tilde{a}}\mD_n$ as above, and this coset is monotypic because the number of $4$-cycles in its permutations depends only on $n$. 
\end{proof} 
\noindent In other words, monotypic panmagic cosets of $\mD_n$ exist in  the same dimensions $n$ as its strictly ternary panmagic cosets, which are, incidentally, always monotypic. Recall that strictly ternary panmagic cosets of $\sD_n$ also exist in the exact same dimensions, by Corollary \ref{PanTerComp}.

\subsection{Monotypic cosets of $\sD_n$}

In square-free dimensions, $\sD_n=\mD_n$ and Theorem \ref{componecycle} describes monotypic cosets of both subgroups. The main result of this section is complementary and negative -- when $n$ is not square-free, monotypic cosets of $\sD_n$ do not exist. 

To see why in a simple case, consider the panmagic coset $\pi_2\sD_{25}$, which splits into five cosets $\pi_a\mD_{25}$ with $a=2, 7, 12, 17, 22$. The orders $\ord_5(a)$ are the same for all of them, as $a\equiv2\pmod{5}$, but while $\ord_{25}(2)=\ord_{25}(12)=\ord_{25}(17)=\ord_{25}(22)=20$, $\ord_{25}(7)=4$. These are the longest cycle lengths of the corresponding $\pi_a$, so $\pi_7$ has the cycle type different from the rest. The next lemma generalizes this observation.
\begin{lemma}\label{q2orders} Let $p$ be a prime and $m$ be an integer such that $p\mid m$ but $p^2\nmid m$. Then there exists exactly one $k$ modulo $p$ with $\ord_{p^2}\big(a+km\big)=\ord_p(a)$, while for all other $k$, $\ord_{p^2}\big(a+km\big)=p\ord_p(a)$. 
\end{lemma}
\begin{proof}
Note that $a+k'm\equiv a+km\pmod{p^2}$ when $k'\equiv k\pmod{p}$ because $pm\equiv0\pmod{p^2}$. In particular, $\ord_{p^2}(a+km)$ depends solely on the residue of $k$ modulo $p$.

Denote $s:=\ord_p(a)$ and consider $\big(a+km\big)^s$. If $(a+k_1m)^s\equiv (a+k_2m)^s\pmod{p^2}$ then expanding the binomial and noticing that $m^2\equiv 0\pmod{p^2}$ gives that 
$$
a^s+sa^{s-1}k_1m\equiv a^s+sa^{s-1}k_2m\pmod{p^2}.
$$
By Fermat's Little Theorem, $s=\ord_p(a)\mid p-1$, so $p\nmid s$ and $p\nmid a$ since $a$ is a unit. Therefore, $k_1m\equiv k_2m\pmod{p^2}$. Since $p\mid m$ but $p^2\nmid m$, $k_1\equiv k_2\pmod{p}$ and the residues of $(a+km)^s$ modulo $p^2$ are distinct for $k=0,\dots,p-1$.

Let us determine these residues. Since $a^s\equiv 1\pmod{p}$, $a^s\equiv 1+l_1p\pmod{p^2}$, and since $p\mid m$, $sa^{s-1}km=l_2p$. Therefore, $(a+km)^s\equiv 1+lp\pmod{p^2}$ for some $0\leq l\leq p-1$. Since there are $p$ distinct residues, every value of $l$ is realized by some $k$. In particular, $l=0$ for some $k_0$ and hence $\big(a+k_0m\big)^s\equiv 1\pmod{p^2}$. But  $\ord_{p^2}(x)=\ord_p(x)$ or $p\ord_p(x)$ for any prime, and $\ord_{p}(a+km)=\ord_p(a)$ for all $k$. Thus, $\ord_{p^2}(a+k_0m)=\ord_p(a)=s$, and for all other $k$ we must have $\ord_{p^2}(a+km)=p\ord_p(a)$. 
\end{proof}
Lemma \ref{q2orders} settles the monotypism question negatively for $n=p^2$, just take $m=p$. But in general, it is not enough to consider only the longest cycles. Still, different cycle types do occur in permutations that are lifts of a single $\pi_a$. Note that $\e_k:=\big(a+k\rad(n)\big)a^{-1}$ is a lifted unity modulo $n$ for any $k$ because $\e_k\equiv 1\pmod{\rad(n)}$. Therefore, $\pi_{a\e_k}=\pi_{a+k\rad(n)}\in\pi_a\sD_n$ for all $k$. When $n$ is not square-free, we will find values of $k$ for which $\pi_{a\e_k}$ have distinct cycle types.
\begin{theorem}\label{SqCycles} Let $2,3\nmid n\geq5$ be a non-square-free integer and $a\in\+{\Z}_n$ be a panmagic unit. Then there are permutations in $\pi_a\sD_n$ with distinct cycle types.
\end{theorem}
\begin{proof}
Since $n$ is not square-free there exist prime factors $p_i$  with $p_i^2\mid n$. We will apply Lemma \ref{q2orders} to each $p_i$ with $m:=\rad(n)$. By definition of the radical, $p_i\mid m$ but $p_i^2\nmid m$ for each $i$. For each $i$, the lemma constructs $k_i$ with $\ord_{p_i^2}(a+k_im)=\ord_{p_i}(a)$. By the Chinese Remainder Theorem, there exists  $k$ with $k\equiv k_i\pmod{p_i}$ for all $i$. Since the order depends only on the residue of $k$ modulo $p_i$, we have
$$
\ord_{{p_i^2}}\big(a+km\big)=\ord_{{p_i^2}}\big(a+k_im\big)=\ord_{p_i}(a)
$$ 
for all $i$. Since $k+1\not\equiv k_i\pmod{p_i}$ for any $i$, by Lemma \ref{q2orders}, $\ord_{p_i^2}(a+(k+1)m)=p_i\ord_{p_i}(a)$ for all $i$. Consider $p_{\min}$ with $\ord_{p_{\min}}(a)$ minimal among all $\ord_{p_i}(a)$. Then,
$$
\ord_{p_i^2}\big((a+(k+1)m\big)=p_i\ord_{p_i}(a)>\ord_{p_{\min}}(a).
$$
Let $d$ be a divisor of $n$ that is not square-free. If $p_i^2\mid d$ then $\ord_{p_i^2}(x)\mid \ord_d(x)$, so  
\begin{equation}\label{orddpmin}
\ord_d\big(a+(k+1)m\big)\geq \ord_{p_i^2}\big(a+(k+1)m\big)>\ord_{p_{\min}}(a).    
\end{equation}

Recall from Theorem \ref{GAcycletypes} that a permutation $\pi_x$ has precisely $\frac{\varphi(d)}{\ord_d(x)}$ cycles of length $\ord_d(x)$ for every $d\mid n$, and no other cycles. 
Therefore, $\pi_{a+(k+1)m}$ can only have cycles of length $\ord_{p_{\min}}(a)$ for some square-free $d\mid n$ because non-square-free $d$ are excluded by \eqref{orddpmin}. But for square-free $d$, $\ord_d(a+(k+1)m)=\ord_d(a+km)$ because $m=\rad(n)\equiv0\pmod{d}$. Therefore, cycles of length $\ord_{p_{\min}}(a)$ in $\pi_{a+(k+1)m}$ for a square-free $d$ are in $1$-$1$ correspondence with cycles of the same length and for the same $d$ in $\pi_{a+km}$ . 

However, $\pi_{a+km}$ also has cycles of length $\ord_{p_{\min}}(a)$ coming from divisors that are not square-free. For example, $\ord_{p_{\min}^2}(a+km)=\ord_{p_{\min}}(a)$, and these cannot have counterparts in $\pi_{a+(k+1)m}$. Thus, there are strictly more cycles of length $\ord_{p_{\min}}(a)$ in $\pi_{a+km}$ than in $\pi_{a+(k+1)m}$ and these two permutations have distinct cycle types.
\end{proof}
To illustrate Theorem \ref{SqCycles}, consider $n=5^2\cdot 7^2$ and take $a=2$, $p_1=5$, $p_2=7$. We first find $k_i$ from $\ord_{p_i^2}(2+35k_i)=\ord_{p_i}(2)$ and $k$ from $k\equiv k_i\pmod{p_i}$. They  are $k_1=3$, $k_2=5$ and $k\equiv 33\pmod{35}$. Since $\ord_{7}(2)=3<4=\ord_{5}(2)$ we have $p_{\min}=7$ and $\ord_{p_{\min}}(a)=3$. Since $\ord_5(2)=4$, we know that divisors $d\mid n$ for cycles of length $3$ can only have $7$ as a prime factor, so $d=7$ or $d=49$. We compute
\begin{align*}
\ord_7(2+35k)&=\ord_7\big(2+35(k+1)\big)=3;\\
\ord_{49}(2+35k)&=3,\ \ \ \ord_{49}\big(2+35(k+1)\big)=21.
\end{align*}
Therefore, by Theorem \ref{GAcycletypes}, $\pi_{2+35k}$ has 
$\frac{\varphi(7)}3+\frac{\varphi(49)}3=16$ cycles of length $3$, while $\pi_{2+35(k+1)}$ has only $\frac{\varphi(7)}3=2$ such cycles, and their cycle types are distinct.



\section{Conclusions and open problems}\label{Conc}

We developed a theory of affine panmagic permutations that focuses on their algebraic and combinatorial properties. Their matrices are the simplest kind of panmagic squares and represent configurations of non-attacking queens on a toroidal chessboard, but it is treating them as permutations that unlocked many new results. We hope that our algebraic approach and its generalizations will be useful for solving other problems. Let us briefly outline some adjacent topics and open problems that lie ahead for those willing to pursue the subject further. This paper focused on the general case of dimensions, affine panmagic permutations for prime dimensions were separately studied in \cite{KoLeeP}.

Theorem \ref{GAcycletypes} describes cycle types not only of panmagic permutations but of all affine permutations with fixed points. However, already dihedral permutations may have no fixed points, e.g. the cyclic shift $\k=\pi_{1,1}$ that has a single cycle of length $n$. One would need to know cycle types of general affine permutations to compute the cycle index of $\Aff(\Z_n)$ used in Polya counting. This seems like a challenging extension of the computation of the cycle index of $\Z_n^\times$ in \cite{BK-V}.

Cycles of affine permutations are known as linear congruential sequences in pseudo-random number generation and they are studied in \cite{Mars}. However, cycle types of multidihedral and lifted dihedral permutations without fixed points do not seem to be known in general. A partial result in this direction, that $\pi_{a,b}$ has a single cycle of length $n$ when $a\equiv1\pmod{p}$ for every $p\mid n$ and $b$ is relatively prime to $n$, is known as the Hull-Dobell theorem \cite{HD,Mars}. In our terminology, $a$ from the Hull-Dobell theorem is none other than a lift of unity modulo $n$, a special case of lifted dihedral units.  

Affine panmagic permutations seem to be closely related to the uniform step method for constructing ``natural" panmagic squares, those filled with natural numbers from $1$ to $n^2$. Planck used them to construct modular $n$-queens solutions already in 1900 \cite{BS2}. The method only works in Polya dimensions, and there are intriguing parallels between Lehmer's results on it \cite{Lehm} and ours. More recently, constructions linking natural panmagic squares to non-affine panmagic permutations were also discovered \cite{BS1}.

Counting all panmagic permutations in dimension $n$ is known to be a hard problem, and even good estimates are hard to come by. It is known that there are more than $2^{n/5}$ of them when $n$ is divisible by a $4k+1$ prime, and, conjecturally, the count is asymptotically $\sim n^{\a n}$ for some $\a>0$ \cite{BS2}. Since there are only $\pfi_3(n)n$ affine panmagic permutations and $\pfi_3(n)\leq n$, `almost all' panmagic permutations are non-affine for large $n$. While some special constructions of such permutations are known \cite{BS2}, they are far from producing the entire set or revealing its algebraic structure. 

Non-affine panmagic permutations also present more complex algebraic and combinatorial questions. For example, are there $N$-ary groups consisting of them? Non-affine panmagic cosets of $D_n$ and $\mD_n$ can be ruled out, but we are not even certain about $\sD_n$. And there may be other groups that work. One would have to find pairs of subgroups of $S_n$, one normal in the other, with non-affine panmagic cosets. Some candidate groups that act on general panmagic permutations are considered in \cite{Eng}. What about the cycle types? Although non-affine panmagic permutations have a fixed point, there is no result like Lemma \ref{FixConj} to conjugate them into a simple permutation like $\pi_a$, so describing their cycle types will also require new ideas.

The linear span of permutation matrices of any $N$-ary panmagic subgroup is an $N$-ary algebra of panmagic squares. Many such examples are provided by our Corollary \ref{primeallstrictarity}. So far, only ternary panmagic algebras have been considered in the literature and they have an appealing alternative description in terms of linear algebra. Consider an invertible matrix $Q$ that commutes with both $P_{\phi}$ and $P_{\k}$, and let $E$ be the square matrix with all entries $1$. Matrices $A$ such that $AQ+QA$ is a scalar multiple of $E$ are called $Q$-regular, and they are panmagic squares that form a ternary algebra \cite{Nord}. Thompson's algebra, the linear span of $\pi_2D_5$, is $Q$-regular with $Q=\frac12I+P_{\k}+P_{\k}^{-1}$. Can such $Q$ be found for other spans of ternary groups from our Corollary \ref{primeallstrictarity}? And conversely, which $Q$-regular algebras are spanned by ternary groups and how can those groups be recovered from $Q$? More generally, what is a linear algebra description of $N$-ary panmagic algebras, and when are they linear spans of $N$-ary panmagic groups?

Although affine panmagic permutations are a small fraction of all of them, this tells us little about their share of the space of panmagic squares. Indeed, there are $n!$ permutations in $S_n$, but the dimension of the linear span of all permutation matrices (which is the space of semimagic squares) is only $(n-1)^2+1$. This means that there are massively many linear relations among permutations matrices, and it is {\it a priori} possible that panmagic permutation matrices (or already affine ones) span the entire space of panmagic squares for Polya $n$. Do they? 

When we consider only squares with positive entries and positive linear combinations, the answer is negative for $n>5$, as proved in \cite{AlKi}. For odd $n$, the space of panmagic squares is known to be $(n-2)^2$-dimensional \cite{Hou}, so we could answer the question if we counted linearly independent panmagic  permutation matrices. Questions about linear relations among matrices of group elements belong to the group representation theory. For the affine case, this suggests looking into representation theory of $\Aff(\Z_n)$ and its subgroups. But this is a task for another day.

\bigskip

\noindent {\bf Funding:} This research did not receive any specific grant from funding agencies in the public, commercial, or not-for-profit sectors.

\end{document}